\documentclass[12pt,reqno]{article}

\usepackage[usenames]{color}
\usepackage{amssymb}
\usepackage{graphicx}
\usepackage{amscd}

\usepackage[colorlinks=true,
linkcolor=webgreen,
filecolor=webbrown,
citecolor=webgreen]{hyperref}

\definecolor{webgreen}{rgb}{0,.5,0}
\definecolor{webbrown}{rgb}{.6,0,0}

\usepackage{color}
\usepackage{fullpage}
\usepackage{float}

\usepackage{epsfig}
\usepackage{graphics,amsmath,amssymb}
\usepackage{amsthm}
\usepackage{amsfonts}
\usepackage{latexsym}
\usepackage{epsf}

\newcommand{\seqnum}[1]{\href{http://www.research.att.com/cgi-bin/access.cgi/as/~njas/sequences/eisA.cgi?Anum=#1}{\underline{#1}}}

\begin{document}


\begin{center}
\vskip 1cm{\LARGE\bf A Generalization of the Binomial Interpolated \\
\vskip .06in
Operator and its Action on Linear\\
\vskip .1in
Recurrent Sequences}
\vskip 1cm
\large
Stefano Barbero, Umberto Cerruti, and Nadir Murru\\
Department of Mathematics \\
University of Turin \\
via Carlo Alberto 8/10 \\
Turin \\
Italy \\
\href{mailto:stefano.barbero@unito.it}{\tt stefano.barbero@unito.it}\\
\href{mailto:umberto.cerruti@unito.it}{\tt umberto.cerruti@unito.it}\\
\href{mailto:nadir.murru@unito.it}{\tt nadir.murru@unito.it}\\
\end{center}

\theoremstyle{plain}
\newtheorem{theorem}{Theorem}
\newtheorem{corollary}[theorem]{Corollary}
\newtheorem{lemma}[theorem]{Lemma}
\newtheorem{proposition}[theorem]{Proposition}

\theoremstyle{definition}
\newtheorem{definition}[theorem]{Definition}
\newtheorem{example}[theorem]{Example}
\newtheorem{conjecture}[theorem]{Conjecture}

\theoremstyle{remark}
\newtheorem{rema}[theorem]{Remark}

\begin{abstract}  
In this paper we study the action of a generalization of the Binomial
interpolated operator on the set of linear recurrent sequences. We find
how the zeros of characteristic polynomials are changed and we prove
that a subset of these operators form a group, with respect to a
well-defined composition law. Furthermore, we study a vast class of
linear recurrent sequences fixed by these operators and many other
interesting properties. Finally, we apply all the results to integer
sequences, finding many relations and formulas involving Catalan
numbers, Fibonacci numbers, Lucas numbers and triangular numbers.
\end{abstract}

\section{Introduction}
The study of operators acting on sequences is a very rich research field, recently developed with the aid of OEIS \cite{Sloane}. Many beautiful connections between apparently unrelated sequences have been discovered in the last years, investigating the properties of some operators like Binomial, Invert and the Hankel transform (see, e.g., the papers of Layman, \cite{Layman}, Prodinger \cite{Prod}, Spivey and Steil \cite{Spivey} and our previous paper \cite{bcm}). With the aim of continuing this exploration, we give in this paper a deep sight on a generalization of the Binomial interpolated operator.
\begin{definition}
We define, over an integral domain $R$, the set $\mathcal{S}(R)$ of sequences $a=(a_n)_{n=0}^{+\infty}$,  and the set $\mathcal{W}(R)$ of linear recurrent sequences.
\end{definition}
\begin{definition} \label{sigma}
The \emph{right-shift operator} $\sigma$ changes any sequence $a\in\mathcal{S}(R)$ as follows:
\begin{equation*} 
\sigma(a)=(a_1,a_2,a_3,\ldots)\quad.
\end{equation*}
\end{definition}
\begin{definition} \label{L}
We recall the definition of the \emph{Binomial interpolated operator $L^{(y)}$}, with parameter $y\in R$, which acts on elements of $\mathcal{S}(R)$ as follows:
$$L^{(y)}(a)=b, \quad b_n=\sum_{i=0}^n\binom{n}{i}y^{n-i}a_i\ .$$
\end{definition}
Prodinger \cite{Prod} defined a generalization of the Binomial interpolated operator as
\begin{equation} \label{new-L} L^{(h,y)}(a)=b,\quad b_n=\sum_{i=0}^n\binom{n}{i}h^iy^{n-i}a_i\ , \end{equation}
for any $a\in\mathcal{S}(R)$ and $h,y\in R$ not zero. This generalization can arise from the study of particular sequences called \emph{variant sequences} and recently introduced by Gould and Quaintance \cite{Gould}. A variant sequence is a sequence $a \in \mathcal{S}(R)$ satisfying the recurrence
\begin{equation} \label{var-seq} a_{n+1}=\sum_{i=0}^n\binom{n}{i}h^iy^{n-i}a_i\ . \end{equation}
These sequences generalize the Bell numbers \seqnum{A000110}, obtainable from (\ref{var-seq}) when $h=y=1$, and the Uppuluri-Carpenter numbers \seqnum{A000587}, obtainable from (\ref{var-seq}) when $h=1$ and $y=-1$. Furthermore, when $h=-1$ and $y\not=0$ the variant sequences are linear recurrent sequences of degree 2. As example, Gould and Quaintance \cite{Gould} have studied the case $h=-1$ and $y=1$ corresponding to the sequence \seqnum{A010892}:
\begin{equation} \label{v-s-1} (b_n)_{n=0}^{+\infty}=(1,1,0,-1,-1,0,1,1,0,-1,-1,0,1,1,\ldots)\ . \end{equation}
In this way it is easy to prove that
$$b_{n+2}=b_{n+1}-b_n \quad \forall n\geq2\ , $$
and other relations. 
By Definition \ref{L}, it is immediate that if $a\in\mathcal{S}(R)$ is a variant sequence with $h=1$
$$L^{(y)}(a)=\sigma(a),$$
i.e., when $h=1$, the operator $L^{(y)}$ shifts a variant sequence of one position.\\
Now if we want an operator with these characteristics for all the variant sequences, we have to consider the operator (\ref{new-L}). Indeed, for any variant sequence $a\in\mathcal S(R)$
$$L^{(h,y)}(a)=\sigma(a).$$
Furthermore, the operator (\ref{new-L}) generalizes three operators introduced by Spivey and Steil \cite{Spivey} in order to relate together several numbers of  sequences. Spivey and Steil \cite{Spivey} have also used them for a new proof of the invariance of the Hankel transform of a sequence under the action of Binomial. These operators convert a sequence $a\in\mathcal S(R)$ into a sequence $b$ as follows:
$$b_n=k^n\sum_{i=0}^n\binom{n}{i}a_i,\quad b_n=\sum_{i=0}^n\binom{n}{i}k^ia_i,\quad b_n=\sum_{i=0}^n\binom{n}{i}k^{n-i}a_i, $$
for some $k\not=0\in R$, but we can consider these operators respectively as
$$b=L^{(k,k)}(a), \quad b=L^{(k,1)}(a), \quad b=L^{(1,k)}(a)=L^{(k)}(a).$$
Thus the operator (\ref{new-L}) surely provides more informations about sequences over $R$ and, in particular, about integer sequences. Prodinger \cite{Prod} has specially worked on the action of this operator over the exponential generating functions of sequences. In the next section we will focus our attention on linear recurrent sequences and their characteristic polynomials, instead of studying the generating functions. Moreover, we will give a composition law between generalized Binomial interpolated operators, when $h \in R$ is an invertible element, which enables us to obtain a group structure.
\section{Action of generalized Binomial interpolated operator over sequences and composition law}
In a previous paper \cite{bcm} we have shown how the Binomial and Invert interpolated operators act on recurrent sequences, finding how these operators change characteristic polynomials. Following the same way of thinking, we now make clear the action of (\ref{new-L}) over the set $\mathcal{W}(R)$.
\begin{theorem} \label{zeros}
Let $a\in\mathcal{W}(R)$ be a linear recurrent sequence of degree $r$, with characteristic polynomial $f(t)$ having zeros $\alpha_1,\ldots,\alpha_r$, over a field $\mathbb F$ containing $R$. Then $L^{(h,y)}(a)$ is a linear recurrent sequence of degree $r$, with characteristic polynomial $g(t)$ having zeros $h\alpha_1+y,\ldots,h\alpha_r+y$. Moreover if
$$f(t)=t^r+\sum\limits_{i = 1}^r {( - 1)^i } \sigma _i t^{r - i},$$ then 
$$g(t)=t^r+\sum\limits_{i = 1}^r {( - 1)^i } \overline {\sigma _i } t^{r - i},$$
where $\sigma_i$ and $\overline{\sigma_i}$ are the symmetric functions of the roots, satisfying the relations
\begin{equation}\label{sigmar}
\overline {\sigma _i }  = \sum\limits_{k = 0}^i {\binom{r-k}{i-k}} h^i y^{i - k} \sigma _i \quad i=1,2,\ldots,r\ .
\end{equation}
\end{theorem}
\begin{proof}
In our previous paper we have proved \cite{bcm} (Theorem 10) that $L^{(y)}(a)$ is a linear recurrent sequence of degree $r$, with characteristic polynomial having zeros $\alpha_1+y,\ldots,\alpha_r+y$. Thus, considering the sequence $b=(h^na_n)_{n=0}^{+\infty}$, with characteristic polynomial $p(t)$ having zeros $h\alpha_1,\ldots,h\alpha_r$, we observe that
$$ L^{(h,y)}(a)=L^{(y)}(b) $$
and so $L^{(h,y)}(a)$ is a linear recurrent sequence of degree $r$, with characteristic polynomial having zeros $h\alpha_1+y,\ldots,h\alpha_r+y$.
Furthermore it is clear the relationship between the symmetric functions $\overline{\overline{\sigma_i}}$ and $\sigma_i$ of the roots of $p(t)$ and $f(t)$, respectively
$$\overline{\overline{\sigma_i}}=h^i\sigma_i.$$
To complete the proof, we only need  to observe that $L^{(y)}(p(t))=p(t-y)$ and that the relation (\ref{sigmar}) clearly hold as consequences, respectively, of Theorem 10 and of Corollary 11 of our previous paper \cite{bcm}.
\end{proof}
We now need a well-defined composition law involving operators of the form (\ref{new-L}). This operation, together with the previous Theorem \ref{zeros}, will enable us to prove many relations between sequences. 
In order to compose the operators $L^{(h,y)}$ with each other, we have to define how the operators acts when $h=0$ or $y=0$. So we pose for all $a\in\mathcal{S}(R)$
$$L^{(0,y)}(a)=(y^na_0)_{n=0}^{+\infty}\ , \quad L^{(h,0)}(a)=(h^na_n)_{n=0}^{+\infty}\ ,\quad L^{(0,0)}(a)=(a_0,0,0,0,\ldots)\ .$$
Under these statements, the composition rule naturally derives from the following
\begin{proposition} \label{comp} Given $h,y,k,w \in R$ we have
$$L^{(h,y)}\circ L^{(k,w)}=L^{(hk,y+wh)}$$
\end{proposition} 
\begin{proof}
For all $a\in\mathcal{S}(R)$
$$(L^{(h,y)}(L^{(k,w)}(a)))_n=\sum_{i=0}^n\binom{n}{i}h^iy^{n-i}\sum_{j=0}^i\binom{i}{j}k^jw^{i-j}a_j=$$
$$=\sum_{j=0}^n\sum_{i=j}^n\binom{n}{j}\binom{n-j}{i-j}h^iy^{n-i}k^jw^{i-j}a_j=$$
$$=\sum_{j=0}^n\binom{n}{j}(hk)^ja_j\sum_{i=j}^n\binom{n-j}{i-j}(hw)^{i-j}y^{n-i}= $$
$$\quad\quad= \sum_{j=0}^n\binom{n}{j}(hk)^j(hw+y)^{n-j}a_j=(L^{(hk,y+wh)}(a))_n\ . $$
\end{proof}
Thus, we have that $L^{(1,0)}$ is the identity, with respect to this composition of operators. Moreover, any operator $L^{(h,y)}$ has an inverse, when $h$ is an invertible element of $R$:
$$L^{(h,y)}\circ L^{(1,0)}=L^{(h,y)}, \quad L^{(1,0)}\circ L^{(h,y)}=L^{(h,y)}$$
$$ L^{(h,y)}\circ L^{(\frac{1}{h},-\frac{y}{h})}=L^{(1,0)}, \quad L^{(\frac{1}{h},-\frac{y}{h})}\circ L^{(h,y)}=L^{(1,0)}.$$
These results are a straightforward proof of the next
\begin{proposition}
Let $\mathcal L$ be the set of the operators $L^{(h,y)}$, where $h$ and $y$ belong to $R$, and $h$ is an invertible element. Then $(\mathcal L,\circ)$ is a group.
\end{proposition}
We now have all the necessary tools to begin a deep exploration of the action of $L^{(h,y)}$ over linear recurrent sequences.  
\section{Fixed sequences and some mutual relations between sequences under the action of generalized Binomial interpolated operator}
We start this section with some interesting results about sequences left unaltered under the action of $L^{(h,y)}$.
First of all we consider a general case, with the only condition that $1-h$ is an invertible element of $R$. In the proof of the following proposition, we will use an umbral approach \`{a} la Gian-Carlo Rota \cite{Rota}. We consider a particular \emph{umbra} $U$ associated to a sequence $a \in \mathcal{S}(R)$, which is a linear functional on $R[z]$ defined by posing 
\begin {equation}\label{udef}
U(z^n)=a_n \quad \forall n=1,2,\ldots \quad  U(1)=a_0 \quad \text{and} \quad U(0)=0,
\end{equation}
 and extended linearly (see Rota \cite{Rota} and Roman \cite{Roman}).
\begin{proposition} \label{umbra}
The sequence
\begin{equation}
a=\left(\left( {\frac{y}{{1 - h}}} \right)^n a_0 \right)_{n=0}^{+\infty},
\end{equation}
is the only sequence satisfying $L^{(h,y)}(a)=a$ for all the possible values of $y \in R$ and $h\in R$ such that $1-h$ is invertible.
\end{proposition}
\begin{proof}
From Definition \ref{new-L} we have $L^{(h,y)}(a)=b$ where, using (\ref{udef})
\begin{equation}
b_n=\sum_{i=0}^n\binom{n}{i}h^iy^{n-i}U(z^i)=U\left(\sum_{i=0}^n\binom{n}{i}h^iy^{n-i}z^i\right)=U\left((hz+y)^n\right).
\end{equation}
The condition $b=a$ corresponds to the equivalent relations
$$U(z^n)=U((hz+y)^n) \quad \forall n=1,2,\ldots, $$
and the linearity of $U$ leads to the equations
$$U(z^n-(hz+y)^n)=0 \quad \forall n=1,2,\ldots, $$
which clearly hold when
$$z^n-(hz+y)^n=0 \quad \forall n=1,2,\ldots .$$
So we immediately obtain 
$$z-(hz+y)=0,$$
and finally $z=\frac{y}{{1 - h}}$, which implies, from (\ref{udef})
$$a_n=U(z^n)=U\left(\left( {\frac{y}{{1 - h}}} \right)^n\right)=\left( {\frac{y}{{1 - h}}} \right)^n U(1)=\left( {\frac{y}{{1 - h}}} \right)^n a_0\ .$$
\end{proof}
\begin{proposition}\label{start1}
Let $a\in \mathcal S(R)$ be a sequence with initial value $a_0=1$, then $L^{(h,y)}(a)=a$ only if $h=1$ and $y=0$ or $h=-1$ and $y=2a_1$.
\end{proposition}
\begin{proof}
By the definition we know that
$$L^{(h,y)}(1,a_1,a_2,\ldots)=(1,a_1h+y,a_2h^2+y(2a_1h+y),\ldots)$$
and the system 
$$\begin{cases} a_1h+y=a_1 \cr a_2h^2+y(2a_1h+y)=a_2  \end{cases}$$
has only the solutions $h=1, y=0$ and $h=-1,y=2a$.
\end{proof}
\begin{rema}
In a previous work Barbero and Cerruti \cite{bc} studied the underlying group structure involving some operators acting over the set of sequences $a \in \mathcal{S}(R)$ such that $a_0=1$. In particular, embedding this set into $R[[t]]$ as follows:
$$\lambda(a)=\sum\limits_{n = 0}^{ + \infty } {a_n t^{n + 1} }\ .$$
and considering the natural composition of series $\circ$, which induce the operation $\bullet$ 
$$\forall a, b \in \mathcal{S}(R)\quad \text{with} \quad a_0=b_0=1 \quad  a\bullet b=\lambda^{-1}(\lambda(a)\circ \lambda(b))\ ,$$
over this set, they have proved that $L^{(y)}(a)=a\bullet X(y)$, where $X(y)=(y^n)_{n=0}^{+\infty}$. Now if we define, as Barbero and Cerruti  \cite{bc}, the operator $\varepsilon$ such that 
$$\forall a \in \mathcal{S}(R) \quad \varepsilon(a)=((-1)^na_n)_{n=0}^{+\infty}\ ,$$ 
we have, for all sequences $a$ with $a_0=1$
$$L^{(-1,y)}(a)=\varepsilon(a)\bullet X(y)\ , $$
and by Proposition \ref{start1}, when the operator $L^{(-1,2a_1)}$ fixes the sequence $a$, we obtain
$$\varepsilon(a)\bullet X(2a_1)=a \ .$$
\end{rema}
Another important result consists in finding all the recurrent sequences of degree 2  fixed by $L^{(h,y)}$, when we make a suitable choice of $h$, $y$ and initial conditions. First of all, we introduce a notation to shortly write linear recurrent sequences of degree 2.
\begin{definition}
We indicate with $a=(a_n)_{n=0}^{+\infty}=\mathcal W(\delta,\gamma,p ,q)$ the linear recurrent sequence of degree 2 with characteristic polynomial $t^2-pt+q$ and initial conditions $\delta$ and $\gamma$, i.e.,
$$\begin{cases} a_0=\delta; \cr
a_1=\gamma; \cr
a_n=pa_{n-1}-qa_{n-2}\quad \forall n\geq2\ .  \end{cases}$$
\end{definition}
\begin{proposition} \label{fixed-seq}
For any linear recurrent sequence  $a=\mathcal W(\delta,\gamma,p ,q)$, we have $L^{(h,y)}(a)=a$ if and only if $(h,y)=(-1,p)$ and $\gamma=\frac{p}{2} \delta$, counting out the trivial case $(h,y)=(1,0)$. Moreover
\begin{equation}\label{relhy} L^{(h,y)}(a)=L^{(-h,y+ph)}(a). \end{equation}
\end{proposition}
\begin{proof}
As we have shown in Theorem \ref{zeros}, when we consider $r=2$ in (\ref{sigmar}), the symmetric functions of the roots of $L^{(h,y)}(f(t))$ are 
$$\overline{\sigma_1}=h\sigma_1+2y \quad \text{and} \quad \overline{\sigma_2}=h^2\sigma_2+ hy\sigma_1+y^2.$$
In our case $\sigma_1=p$ and $\sigma_2=q$, so the characteristic polynomial of $L^{(h,y)}(a)$ will coincide with the one of $a$ if and only if 
$$p=hp+2y \quad \text{and}\quad q=h^2q+hyp+y^2.$$
From the first of these equalities we find $y=\frac{1-h}{2}p$. Substituting into the second one, we obtain
$$(1-h)(1+h)(q-\frac{p}{4})=0,$$
which is true for all $p$ and $q$ if and only if $h=-1$ or $h=1$. The case $h=1$ gives $y=0$, not so interesting, while
the case $h=-1$ gives $y=p$. Finally, if we pay attention to the initial conditions, we have 
$$L^{(-1,p)}(a)=(\delta,-\gamma+p\delta,\ldots)$$
so $-\gamma+p\delta=\gamma$ if and only if $\gamma=\frac{p}{2}\delta$. The equality (\ref{relhy}) is a straightforward consequence of Proposition \ref{comp}.
\end{proof}
\begin{example}\label{Lucas}
As an example of sequence like those studied in Proposition \ref{fixed-seq}, we can consider the Lucas numbers \seqnum{A000032} $$l=(l_n)_{n=0}^{+\infty}=(2, 1, 3, 4, 7, 11, 18, 29, 47, 76, 123,\ldots)\ .$$
Since this sequence corresponds to $l=\mathcal{W}(2,1,1,-1)$, we immediately observe that
$$L^{(-1,1)}(l)=l.$$
\end{example}
Let us consider any $a=\mathcal{W}\left(\delta,\frac{p}{2}\delta,p,q\right)$ satisfying the hypotheses of Proposition \ref{fixed-seq}. Applying $L^{(-1,p)}$, we directly obtain the following interesting identities
\begin{equation}\label{fixed-seq-formula1} a_n=\sum_{i=0}^n\binom{n}{i}(-1)^ip^{n-i}a_i\ , \end{equation}
and 
\begin{equation} \label{fixed-seq-formula2}
\sum_{i=0}^{n-1}\binom{n}{i}(-1)^ip^{n-i}a_i=\begin{cases} 0, & \text{if } n
\ \text{even}; \\
2a_n, & \text{if } n \ \text{odd};
\end{cases}\end{equation}
which we will use in the next section, finding some new applications to well-known integer sequences.
Furthermore, as a direct consequence of Proposition \ref{fixed-seq}, we have the following
\begin{corollary} \label{gen-fig-num}
Let $a=(a_n)_{n=0}^{+\infty}$ be a linear recurrent sequence with characteristic polynomial $(t^2-pt+q)^m$. Then $L^{(-1,p)}(a)$ has the same characteristic polynomial.
\end{corollary}
 
We study another interesting property of the action of $L^{(h,y)}$ on linear recurrent sequences of degree 2 in the next
\begin{theorem} \label{mult}
If $a=\mathcal{W}(\delta,\gamma,p,q)$ and $u=\mathcal{W}(0,1,p,q)$ then
$$b=(a_{kn})_{n=0}^{+\infty}=L^{(u_k,-qu_{k-1})}(a), \quad k\geq 1.$$
\end{theorem}
\begin{proof}
If  $t^2-pt+q$ has zeros $\alpha$ and $\beta$, then $b$ recurs with characteristic polynomial $t^2-v_kt+q^k$ whose zeros are $\alpha^k$ and $\beta^k$, and $v_k$ is the $k$-th term of $v=\mathcal{W}(2,p,p,q)$. Indeed, the characteristic polynomial of $b$ is the characteristic polynomial of the matrix \cite{Cerruti}:
$$\begin{pmatrix} 0 & 1 \cr -q & p  \end{pmatrix}^k.$$
If we take $h=u_k=\frac{\alpha^k-\beta^k}{\alpha-\beta}$ and $y=-qu_{k-1}=\frac{\alpha\beta^k-\alpha^k\beta}{\alpha-\beta}$, by Theorem \ref{zeros} the zeros of the characteristic polynomial of $b$ are 
$$\cfrac{\alpha^k-\beta^k}{\alpha-\beta}\cdot\alpha+\cfrac{\alpha\beta^k-\alpha^k\beta}{\alpha-\beta}=\alpha^k,\quad \cfrac{\alpha^k-\beta^k}{\alpha-\beta}\cdot\beta+\cfrac{\alpha\beta^k-\alpha^k\beta}{\alpha-\beta}=\beta^k$$
and the initial conditions are $a_0$ and $a_k$.
\end{proof}

Finally, we see how the Hankel transform \cite{Layman} changes after applying $L^{(h,y)}.$ 
\begin{proposition}\label{hankel}
Let $H$ be the Hankel transform, for any $a\in\mathcal{S}(R)$
$$H(L^{(h,y)}(a))=(h^{n(n+1)}k_n)_{n=0}^{+\infty}.$$
where $(k_n)_{n=0}^{+\infty}=H(a)$.
\end{proposition}
\begin{proof}
$$H(L^{(h,y)}(a))=H(L^{(1,y-1)}(L^{(h,1)}(a)))=H(L^{(h,1)}(a))=(h^{n(n+1)}k_n)_{n=0}^{+\infty}$$
since Spivey and Steil \cite{Spivey} have proved the last equality and the Hankel transform is invariant under the Binomial interpolated operator \cite{Spivey}.
\end{proof}
The last Proposition \ref{hankel} allows us to find the Hankel transform of some sequences of the form $L^{(h,y)}(a)$, starting from a known Hankel transform $H(a)$. For example the Hankel transform of a constant sequence $a=(c,c,c,\ldots)$ is
$$H(a)=(c,0,0,\ldots), \quad \forall c\in R\ ,$$
since 
$$L^{(h,y)}(a)=(c(h+y)^n)_{n=0}^{+\infty}\ ,$$
we have
$$H((c(h+y)^n)_{n=0}^{+\infty})=(c,0,0,\ldots)\ .$$
\section{Applications to integer sequences}
In the previous section we have shown many properties about the action of $L^{(h,y)}$ on linear recurrent sequences over $R$. In this section we focus our attention on integer sequences and we will see how the operator $L^{(h,y)}$ is very useful in order to find many new relations and informations. First of all, we start showing how, by Theorem \ref{zeros}, we can immediately prove a relation between integer sequences only conjectured by Spivey and Steil \cite{Spivey} (p. 11, Table 1).
\begin{proposition}
Let $a=\mathcal W(1,5,6,1)$ and $b=\mathcal{W}(1,3,4,2)$
then
$$L^{(\frac{1}{2},\frac{1}{2})}(a)=b\ .$$
\end{proposition}
\begin{proof}
The sequence $a$ is the linear recurrent sequence of integers with initial conditions $a_0=1$, $a_1=5$ and characteristic polynomial $t^2-6t+1$ \seqnum{A001653}: 
$$a=(1,5,29,169,985,\ldots)\ .$$
The sequence $b$ is the linear recurrent sequence of integers with initial conditions $b_0=1$, $b_1=3$ and characteristic polynomial $t^2-4t+2$ \seqnum{A007052} in OEIS:
$$b=(1,3,10,34,116,\ldots)\ .$$
The zeros of $t^2-6t+1$ are $3\pm2\sqrt{2}$. The operator $L^{(\frac{1}{2},\frac{1}{2})}$ transforms $a$ into a sequence with characteristic polynomial whose zeros are 
$$\cfrac{1}{2}(3\pm2\sqrt{2})+\cfrac{1}{2}=2\pm\sqrt{2},$$
corresponding to the zeros of $t^2-4t+2$. Finally, it is straightforward to check that $L^{(\frac{1}{2},\frac{1}{2})}$ changes initial conditions of $a$ into those of $b$.
\end{proof}
Many relations involving integer sequences can be found using the Proposition \ref{fixed-seq} and the derived formulas (\ref{fixed-seq-formula1}) and (\ref{fixed-seq-formula2}). For Lucas numbers \seqnum{A000032} we have seen in Example \ref{Lucas} that 
$$L^{(-1,1)}(l)=l\ .$$
We also have the following new relations
$$ \sum_{i=0}^{2k-1}\binom{2k}{i}(-1)^il_i=0, $$
$$ \cfrac{1}{2}\sum_{i=0}^{2k}\binom{2k+1}{i}(-1)^il_i=l_{2k+1}\ . $$
The Proposition \ref{fixed-seq} can be applied to many other interesting recurrent sequences of integers. For example, let $a$ be the sequence of numerators of continued fraction convergents to $\sqrt{2}$ \seqnum{A001333}. This is a linear recurrent sequence of degree 2, in particular, with respect to our notation, it is $a=\mathcal{W}(1,1,2,-1)$, i.e., 
$$(1, 1, 3, 7, 17, 41, 99, 239, 577, 1393,\ldots)\ .$$
Its characteristic polynomial and its initial conditions satisfy the hypotheses of Proposition \ref{fixed-seq} and we have
$$L^{(-1,2)}(a)=a$$
and
$$\sum_{i=0}^{n-1}\binom{n}{i}(-1)^i2^{n-i}a_i=\begin{cases} 0,  &
\text{if } n \ \text{even}; \\
2a_n, & \text{if } n \  \text{odd}. \end{cases}$$
As we have already seen, not only recurrent sequences can be fixed by $L^{(h,y)}$. In the next Theorem we see a surprising result involving the famous Catalan numbers \seqnum{A000108}. They are fixed by $L^{(-1,4)}$.
\begin{theorem}
Let $C=(C_n)_{n=0}^{+\infty}$ be the sequence of the Catalan numbers 
$$C=(1, 1, 2, 5, 14, 42, 132, 429, 1430, 4862, 16796,\ldots)$$
then
$$L^{(-1,4)}(\sigma(C))=\sigma(C).$$
\end{theorem}
\begin{proof}
Recalling the explicit formula of Catalan numbers, we have
$$C_n=\binom{2n}{n}\cfrac{1}{n+1},\quad \forall n\geq0 \ .$$
Now we pose
$$L^{(-1,4)}(\sigma(C))=a\ ,$$
and we have
$$a_n=\sum_{h=0}^n\binom{n}{h}(-1)^h4^{n-h}\binom{2h+2}{h+1}\cfrac{1}{h+2}=\sum_{h=0}^nF(h,n),\quad \forall n\geq0\ .$$
We can prove that
\begin{equation} \label{fr}2(2n+3)F(h,n)+(-n-3)F(h,n+1)=\Delta(F(h,n)R(h,n)),\end{equation}
for all $n\geq0$, where $R(h,n)=\cfrac{4h(h+2)}{-h+n+1}$ and $\Delta$ is the forward difference operator in $h$. Indeed, dividing the last equation by $F(h,n)$ we obtain for the first member
$$2(2n+3)+4(-n-3)\cfrac{n+1}{n+1-h}=-\cfrac{6+6h+6n+4hn}{n+1-h}$$
and for the second member
$$\cfrac{F(h+1,n)R(h+1,n)}{F(h,n)}-R(h,n)=-2(2h+3)-\cfrac{4h(h+2)}{-h+n+1}=$$ $$=-\cfrac{6+6h+6n+4hn}{n+1-h}.$$
Now summing both members of the equation (\ref{fr}) over $h$ , from $0$ to $n-1$, and rearranging the resulting terms with a little bit of calculation, it is easy to see that we have
$$2(2n+3)a_n+(-n-3)a_{n+1}=0,\quad \forall n\geq0 \ ,$$
and it is well-known (see, e.g., in OEIS \cite{Sloane} the sequence \seqnum{A000108}) that 
$$2(2n+3)C_{n+1}+(-n-3)C_{n+2}=0, \quad \forall n\geq0\ ,$$
i.e., we have proved that $a=\sigma(C)$.
\end{proof}
As consequence we have a new recurrent formula for Catalan numbers:
$$C_{n+1}=\sum_{h=0}^n\binom{n}{h}(-1)^h4^{n-h}C_{h+1}\ ,$$
which we can decompose into two new formulas, like in equation (\ref{fixed-seq-formula2})
$$C_n=\cfrac{1}{4n}\sum_{h=0}^{n-2}\binom{n}{h}(-1)^h4^{n-h}C_{h+1},\quad n \ \hbox{even}$$
$$C_{n+1}=\cfrac{1}{2}\sum_{h=0}^{n-1}\binom{n}{h}(-1)^h4^{n-h}C_{h+1},\quad n \ \hbox{odd}\ .$$
It is really interesting to observe that when, e.g., $n=2k$ in the first formula, we have only to choose $n=2k-1$ in the second formula in order to have two different new representations of the same Catalan number $C_{2k}$.
In the next Propositions, where we find a beautiful formula for triangular numbers \seqnum{A000217} and for the binomial $\binom{n}{4}$, we apply the Corollary \ref{gen-fig-num}.
\begin{proposition}
Let $T=(T_n)_{n=0}^{+\infty}$ be the triangular numbers 
$$(0,1,3,6,10,15,\ldots)\ ,$$
then
$$T_n=\sum_{i=0}^n\binom{n}{i}(-1)^i2^{n-i}T_i\ .$$
\end{proposition}
\begin{proof}
Let $T'=(T'_n)_{n=0}^{+\infty}$ be the sequence
$$(0,0,1,3,6,10,15,\ldots)\ ,$$
i.e., it is the sequence of triangular numbers starting from $(0,0,1,\ldots)$ instead of  $(0,1,\ldots)$. It is easy to check that $T'$ recurs with characteristic polynomial $(t-1)^4=(t^2-2t+1)^2$. Thus, by Corollary \ref{gen-fig-num}, the sequence $L^{(-1,2)}(T')$ has the same characteristic polynomial. Furthermore, in general
$$L^{(-1,p)}(a_0,a_1,a_2,a_3,\ldots)=(a_0,\cfrac{a_0p}{2},a_2,\cfrac{6a_2p-a_0p^3}{4},\ldots)\ ,$$
i.e., in our case the initial conditions of $T'$ are fixed by $L^{(-1,2)}$, and clearly
$$L^{(-1,2)}(T')=T'\ .$$
If we observe that $T=\sigma(T')$ and $T'_0=0$, recalling (\ref{new-L}), we have
$$T_n=\sum_{i=0}^n\binom{n}{i}(-1)^i2^{n-i}T_i\ .$$
\end{proof}
\begin{proposition}
For any integer $n\geq0$ we have
$$\binom{n}{4}=\sum_{i=0}^n\binom{n}{i}\binom{i}{4}(-1)^i2^{n-i}.$$
\end{proposition}
\begin{proof}
Let $a$ be the linear recurrent sequence with characteristic polynomial $(t-1)^6$ and initial conditions $(0,0,0,0,1,5)$. It is possible to verify that 
$$a_n=\binom{n}{4},\quad \forall n\geq0\,$$
i.e., $a$ is the sequence \seqnum{A000332}. By Corollary \ref{gen-fig-num} and checking that initial condition are fixed by $L^{(-1,2)}$, we have
$$L^{(-1,2)}(a)=a\ ,$$
and we can explicitly write
$$\binom{n}{4}=\sum_{i=0}^n\binom{n}{i}\binom{i}{4}(-1)^i2^{n-i}.$$
\end{proof}
Finally, we see some applications of the Theorem \ref{mult} to the Fibonacci numbers \seqnum{A000045} $F=(F_n)_{n=0}^{+\infty}=\mathcal{W}(0,1,1,-1)$. Applying directly this theorem, we have
$$(F_{kn})_{n=0}^{+\infty}=L^{(F_k,F_{k-1})}(F)\ ,$$
and we find the following known formula \cite{Sloane}:
$$F_{kn}=\sum_{i=0}^n\binom{n}{i}F_k^iF_{k-1}^{n-i}F_i\ .$$
But if we apply the Theorem \ref{mult} to the Lucas numbers $l$, we have the following relation between Fibonacci and Lucas numbers:
$$(l_{kn})_{n=0}^{+\infty}=L^{(F_k,F_{k-1})}(l)\ ,$$
which gives
$$l_{kn}=\sum_{i=0}^n\binom{n}{i}F_k^iF_{k-1}^{n-i}l_i\ .$$

\section{Final remarks}
In this article we focused on the study of sequences fixed by generalized Binomial interpolated operator. In particular, we have seen this for linear recurrent sequences of degree 2 and this result could be extended to linear recurrent sequences of higher degree. However, as we have already seen, there are fixed points of $L^{(h,y)}$ which are not linear recurrent sequences. A future development of this article is the general study of this problem. It would seem that, except for the particular case studied in Proposition \ref{umbra}, $L^{(h,y)}$ fixes sequences only when $h=1$ (and consequently $y=0$, i.e., we have the identity) or when $h=-1$.

\bigskip
\hrule
\bigskip
\noindent 2000 {\it Mathematics Subject Classification}: Primary 11B37;
Secondary 11B39.

\noindent \emph{Keywords:}  binomial operator, Catalan numbers,
Fibonacci numbers, Lucas numbers, 
triangular numbers, recurrent sequences.

\bigskip
\hrule
\bigskip

\noindent (Concerned with sequences 
\seqnum{A000032},
\seqnum{A000045},
\seqnum{A000108},
\seqnum{A000110},
\seqnum{A000217},
\seqnum{A000332},
\seqnum{A000587},
\seqnum{A001333},
\seqnum{A001653},
\seqnum{A007052}, and
\seqnum{A010892}.)

\bigskip
\hrule
\bigskip

\vspace*{+.1in}
\noindent
Received July 30 2010;
revised version received   December 6 2010.
Published in {\it Journal of Integer Sequences}, December 8 2010.

\bigskip
\hrule
\bigskip

\noindent
Return to
\htmladdnormallink{Journal of Integer Sequences home page}{http://www.cs.uwaterloo.ca/journals/JIS/}.
\vskip .1in


\begin{thebibliography}{99}

\bibitem{bc} S. Barbero and U. Cerruti, Catalan moments on the occasion the Thirteenth International Conference on Fibonacci Numbers and Their Applications, \emph{Congr. Numer.} \textbf{201} (2010), 187--209.

\bibitem{bcm} S. Barbero , U. Cerruti, and N. Murru, Trasforming recurrent sequences by using Binomial and Invert operators, \emph{J. Integer Seq.} \textbf{13} (2010), \href{http://www.cs.uwaterloo.ca/journals/JIS/VOL13/Barbero/barbero5.pdf}{Article 10.7.7}.

\bibitem{Cerruti} U. Cerruti and F. Vaccarino, Matrices, recurrent sequences and arithmetic, \emph{Applications of Fibonacci Numbers} \textbf{6} (1996), 53--62.

\bibitem{Gould} H. W. Gould and J. Quaintance, Bell numbers and variant sequences derived from a general functional differential equation, \emph{Integers} \textbf{9(A44)} (2009), 581--589.

\bibitem{Layman} J. W. Layman, The Hankel transform and some of its properties, \emph{J. Integer Seq.} \textbf{4} (2001), \href{http://www.cs.uwaterloo.ca/journals/JIS/VOL4/LAYMAN/hankel.pdf}{Article 01.1.5}.

\bibitem{Prod} H. Prodinger, Some information about the binomial transform, \emph{Fibonacci Quart.} \textbf{32(5)} (1994), 412--415.

\bibitem{Sloane} N. J. A. Sloane, \emph{The On--Line Encyclopedia of Integer Sequences}. Published electronically
at, \href{http://www.research.att.com/~njas/sequences}{http://www.research.att.com/~njas/sequences}, 2010.

\bibitem{Spivey} M. Z. Spivey and L. L. Steil, 
The $k$--binomial transforms and the Hankel transform, \emph{J. Integer Seq.} \textbf{9} (2006), \href{http://www.cs.uwaterloo.ca/journals/JIS/VOL9/Spivey/spivey7.pdf}{Article 06.1.1}.

\bibitem{Rota} S. M. Roman and G.--C. Rota, The Umbral calculus, \emph{Adv. Math}. \textbf{27} (1978), 95--188. 

\bibitem{Roman}S. M. Roman, \emph{The Umbral Calculus}, Academic Press, 1984.

\end{thebibliography}
\end{document}